\documentclass[12]{amsart}
\usepackage{titletoc}
\usepackage{blindtext}
\usepackage{amsfonts,amssymb,amsmath,amsthm}

\newtheorem{innercustomgeneric}{\customgenericname}
\providecommand{\customgenericname}{}
\newcommand{\newcustomtheorem}[2]{%
  \newenvironment{#1}[1]
  {%
   \renewcommand\customgenericname{#2}%
   \renewcommand\theinnercustomgeneric{##1}%
   \innercustomgeneric
  }
  {\endinnercustomgeneric}
}

\newcustomtheorem{customthm}{Theorem}
\newcustomtheorem{customlemma}{Lemma}

\usepackage{lineno}

\usepackage{graphicx}

\usepackage{float, tabularx, booktabs, caption}

\usepackage{epsfig}

\usepackage{graphicx, color}

\makeatletter
\@namedef{subjclassname@2020}{\textup{2020} Mathematics Subject Classification}
\makeatother

\makeatletter

\newtheorem*{thrm}{Theorem}
\newtheorem{lem}{Lemma}

\numberwithin{equation}{section}

\begin{document}
\title[Growth of periodic orbits for star flows]{At most exponential growth of periodic orbits for star flows}

\author[C.A. Morales]{C.A. Morales}
\address{Hangzhou International Innovation Institute of Beihang University,  
Hangzhou 311115, China.}
\email{tchivatze@gmail.com}

\author{E. Rego}
\address{AGH University of Science and Technology
Krakow, Poland.}
\email{rego@agh.edu.pl}

\author{X. Wen}
\address{Key Laboratory of Mathematics, Informatics and Behavioral Semantics,  Institute of Artificial Intelligence $\&$ School of Mathematical Sciences,
Beihang University,
Beijing 1000191, Peoples Republic of China.}
\email{wenxiao@buaa.edu.cn}


\begin{abstract}
We prove that on any closed manifold there exists a $C^1$
open and dense subset of star flows for which the periodic orbits grow at most exponentially. We also explain how our methods apply to the sectional-hyperbolic attractors including the Lorenz polynomial equation.
\end{abstract}

\maketitle




\section{Introduction}
\label{sec1}

\noindent
The growth of periodic orbits has been extensively studied over the past decades. In fact, Artin and Mazur \cite{am} proved that it is at most exponential for a $C^k$-dense subset of $C^k$ endomorphisms of compact manifolds, for $k \geq 1$. These authors also conjectured that this property extends to diffeomorphisms and vector fields. Meyer \cite{m} established at most exponential growth for $C^2$ diffeomorphisms and vector fields with hyperbolic nonwandering sets, and Shub later reduced this conclusion to the $C^1$ setting. Kaloshin \cite{ka} settled the Artin--Mazur conjecture for diffeomorphisms, while the vector field case remains open, to the best of our knowledge. Nevertheless, several related results are known for flows and semiflows: at most exponential growth holds for expansive flows \cite{bw}, $N$-expansive flows \cite{lms}, and eventually expansive semiflows \cite{lr}. On the other hand, a recent result of Wu et al.~\cite{wyz} shows that, for $C^1$-generic vector fields on closed manifolds, the growth of periodic orbits is at least exponential. Further results about the growth of periodic orbits can be found in \cite{szz}.

In this paper, we study this growth for {\em star flows}, that is, $C^1$ vector fields on a closed manifold that cannot be $C^1$-approximated by vector fields exhibiting nonhyperbolic closed orbits. Such vector fields play a fundamental role in the modern theory of dynamical systems; see, for instance, \cite{bdl,gw}. Our main result shows that there exists a $C^1$ open and dense subset of star flows for which the growth of periodic orbits is at most exponential. Thus, while generic vector fields exhibit at least exponential growth, within the class of star flows the generic behavior is at most exponential.
We also explain how our methods apply to prove at most exponential growth of periodic orbits for the sectional-hyperbolic attractors \cite{mm} including the Lorenz polynomial equation with classical parameters \cite{l}.
We now state our result precisely.

Consider a {\em closed} (i.e. compact connected boundaryless Riemannian) manifold $M$.
Note that the set of $C^1$ vector fields of $M$ is a Banach space. Fix such a vector field $X$ and denote by $\varphi_t$ the flow induced by $X$.
A point $x\in M$ is {\em periodic} if there is a minimal $\tau=\tau(x)$ (called period) such that $\varphi_\tau(x)=x$.
A {\em periodic orbit} is the orbit $\{\phi_t(x)\mid t\in\mathbb{R}\}$ of a periodic point $x$, and its period is that of $x$.
Denote by $N_t(X)$ the number of periodic orbits of $X$ with period less than or equal to $t$.
We say that $N_t(X)$ {\em grows at most exponentially} if there is $c=c(X)<\infty$ such that
$$
N_t(X)\leq c^t,\qquad\forall t>0.
$$
Our result is the following:

\begin{thrm}
There is a $C^1$ open and dense $\mathcal{O}$ subset of star flows of $M$ such that if $X\in \mathcal{O}$, then $N_t(X)$ grows at most exponentially.
\end{thrm}

In analogy with Theorem 2 in \cite{m} (or \cite{s}), where the convergence of the zeta function for vector fields with hyperbolic nonwandering sets was established, the above theorem guarantees the convergence of the zeta function on a $C^1$ open and dense subset of star flows on any closed manifold.
In Section \ref{sec2}, we prove the theorem. In the final remarks (Section \ref{sec2}) we explain how to apply our methods to the sectional-hyperbolic attractors and the Lorenz equation.

\section{Proof of the theorem}
\label{sec2}

\noindent
The proof is based on our unpublished manuscript \cite{mrw}.

Hereafter we fix a closed manifold $M$, a $C^1$ vector field $X$ of $M$ and a compact invariant set $\Lambda$ of $X$.
This means that
$\varphi_t(\Lambda)=\Lambda$ for all $t\in\mathbb{R}$.
We say that $x\in M$ is a {\em singularity} if $X(x)=0$.
Denote by $Sing(X)$ the set of singularities and
$\Lambda^*=\Lambda\setminus Sing(X)$.
We further assume $\Lambda^*\neq\emptyset$.

Denote by $\|\cdot\|$ and $d$ the norm and the distance of the tangent bundle $TM$ and $M$ generated by the Riemannian metric respectively.

Define
\[
d_t(x,y)
\;=\;
\sup_{0 \le s \le t} \frac{d\bigl(\varphi_s(x),\,\varphi_s(y)\bigr)}{\|X(\varphi_s(x))\|}
\quad\quad (t \ge 0,\, x\in M^*,\, y\in M)
\]
and
\[
B(x,t,\epsilon)
\;=\;
\{\,y \in M : d_t(x,y) < \epsilon\}
\quad\quad(\forall\, x \in M^*).
\]

We have the following

\begin{lem}
\label{symmetry}
There is $r_0>0$ such that if $t\geq0$, $0<\epsilon<r_0$, $x\in M^*$ and
$$
y\in B(x,t,\frac{\epsilon}4)\quad\Longrightarrow\quad
B(x,t,\frac{\epsilon}4)\subset B(y,t,\epsilon).
$$
\end{lem}

\begin{proof}
By Lemma 2.1 in \cite{wy} there is $r_0>0$ such that
if $a\in M^*$ and $b\in M$ then
\begin{equation}
\label{metanfeta}
d(a,b)\leq r_0\|X(a)\|\quad\implies\quad \frac{1}2\|X(a)\|\leq \|X(b)\|\leq 2\|X(a)\|.
\end{equation}

Take $t\geq0$, $0<\epsilon<r_0$, $x\in M^*$ and
$y\in B(x,t,\frac{\epsilon}4)$.
Then, $\|X(y)\|\geq \frac{1}2\|X(x)\|>0$ and so $y\in M^*$.
Now, take $t\geq0$, $0<\epsilon<r_0$ and $w\in B(x,t,\frac{\epsilon}4)$.
Since $y\in B(x,t,\frac{\epsilon}4)$, $\|X(\varphi_s(x))\|\leq 2\|X(\varphi_s(y))\|$ for all $0\leq s\leq t$ by \eqref{metanfeta} thus
$$
d(\varphi_s(y),\varphi_s(x))<\frac{\epsilon}2\|X(\varphi_s(y))\|,\quad\quad\forall 0\leq s\leq t.
$$
On the other hand, $w\in B(x,t,\frac{\epsilon}4)$ so $d(\varphi_s(x),\varphi_s(w))<\frac{\epsilon}4\|X(\varphi_s(x))\|$ for all $0\leq s\leq t$ thus
$$
d(\varphi_s(x),\varphi_s(w))<\frac{\epsilon}2\|X(\varphi_s(y))\|,\quad\quad\forall 0\leq s\leq t.
$$
Then,
\begin{eqnarray*}
d(\varphi_s(y),\varphi_s(w))&\leq& d(\varphi_s(y),\varphi_s(x))+d(\varphi_s(x),\varphi_s(w))\\
&<&\frac{\epsilon}2\|X(\varphi_s(y))\|+\frac{\epsilon}2\|X(\varphi_s(y))\|\\
&=&\epsilon\|X(\varphi_s(y))\|,\quad\quad\forall 0\leq s\leq t,
\end{eqnarray*}
proving $w\in B(y,t,\epsilon)$. Therefore, $B(x,t,\frac{\epsilon}4)\subset
B(y,t,\epsilon)$ and we are done.
\end{proof}

Given $\epsilon>0$ write $\Lambda_\epsilon=\{x\in \Lambda:\|X(x)\|\geq\delta\}$.
Call \(F \subset M\) {\em \((t,\epsilon,\Lambda)\)-spanning set} if \(F \subset M^*\) and
\[
\Lambda_\epsilon \;\subset\; B(F,t,\epsilon)
\quad\text{where}\quad
B(F,t,\epsilon)\;=\;\bigcup_{x \in F} B(x,t,\epsilon).
\]
Clearly \(\Lambda_\epsilon\) is compact and \(B(x,t,\epsilon)\) is an open neighborhood of \(x\) for all \(x \in M^*\). So, the number
\[
R(t,\epsilon,\Lambda)
\;=\;
\min \,\bigl\{\,
\mathrm{card}(F)
:\,F \text{ is a  \((t,\epsilon,\Lambda)\)-spanning set}
\bigr\}
\]
is finite for all \(t \ge 0\) and \(\epsilon > 0\), where \(\mathrm{card}(F)\) denotes the cardinality of \(F\).

Following \cite{r} we define
\[
E(X,\Lambda)
\;=\;
\lim_{\epsilon \to 0}\,\limsup_{t \to \infty}
\frac{1}{t} \log R(t,\epsilon,\Lambda).
\]
Equivalently,
\begin{equation}
\label{fria}
E(X,\Lambda)=\sup_{K}E(X,K),
\end{equation}
where the supremum is over the compact subsets $K\subset \Lambda^*$. This number is called {\em rescaled topological entropy} in \cite{mrw}. We shall use some properties of this number.

\begin{lem}
\label{perseguido}
 $E(X,\Lambda)<\infty$.
\end{lem}

\begin{proof}
It suffices to prove
\begin{equation}
\label{pi}
E(X,\Lambda)\leq 2d L,
\end{equation}
where $L$  is a Lipschitz constant of $X$ and $d$ is the dimension of $M$.

Indeed, fix a Lipschitz constant $L$.
By Lemma 2.5 in \cite{hw} one has
\begin{equation}
\label{cope}
e^{-Ls}\leq \frac{\|X(\varphi_s(z))\|}{\|X(z)\|}\leq e^{sL},\quad\quad\forall z\in M^*,\, s\geq0.
\end{equation}
If $L=0$, the right-hand side of \eqref{pi} is zero
while $\|X(\varphi_s(z))\|=\|X(z)\|$ for all $z\in M^*$ and $s\geq0$ by \eqref{cope}.
Then, for every compact $K\subset \Lambda^*$ and $\epsilon>0$, every  $(0,\epsilon,K)$-spanning set is
 $(t,\epsilon,K)$-spanning ($\forall t\geq0$) so $E(X,K)=0$
for all compact $K\subset M^*$. Thus, $E(X)=0$ by \eqref{fria} and
\eqref{pi} holds.
Therefore, we can assume $L>0$.

Now, fix a compact subset $K\subset \Lambda^*$ and $\epsilon>0$.
Then, there are diffeomorphisms
$f_1,\cdots, f_n:B(0,2)\to M$ and a positive number $A$ such that
$$
K\subset \bigcup_{i=1}^nf_i(B(0,1))\subset \bigcup_{i=1}^nf_i(B(0,2))\subset M^*
\quad\mbox{ and }\quad
d(f_i(u),f_i(v))\leq A\|u-v\|,
$$
for all $u,v\in B(0,2)$ and $1\leq i\leq n$.
The third of the above inclusions implies that there is $\rho\in (0,\infty)$ such that
$$
\rho\leq \|X(x)\|,\quad\quad\forall x\in \bigcup_{i=1}^nf_i(B(0,2)).
$$

Now, for each $0<\delta\leq 2$ we let
$$
\mathcal{E}(\delta)=\{(\delta l_1,\cdots,\delta l_d)\in \mathbb{R}^d\mid l_i\in \mathbb{Z},\, |l_i\delta|<2\}.
$$
Then, $card(\mathcal{E}(\delta))\leq (\frac{5}\delta)^d$ and there is a constant $B>0$ (depending on the Euclidean metric of $\mathbb{R}^d$) such that for every $v\in B(0,1)$ there is $u\in E(\delta)$ satisfying $\|u-v\|\leq B\delta$. Replacing $B$ by $\frac{B}\rho$ we have
$\|u-v\|\leq B\rho\delta$ for all such $u,v$.

Since $L>0$, we can choose $T>0$ large such that
\begin{equation}
\label{pobre}
\frac{\epsilon}{e^{2Lt}AB}<2,\quad\quad\forall t\geq T.
\end{equation}
For all $t\geq T$ we define
$$
F=\bigcup_{i=1}^n f_i(E(\frac{\epsilon}{e^{2Lt}AB})).
$$
It follows that
\begin{eqnarray*}
card(F)&\leq & \sum_{i=1}^n card(f_i(E(\frac{\epsilon}{e^{2Lt}AB})))\\
&\overset{\eqref{pobre}}{\leq}& \left(\frac{5}{\left(\frac{\epsilon}{e^{2Lt}AB}\right)}\right)^dn\\
&=&
\left(\frac{5e^{2Lt}AB}{\epsilon}\right)^dn\\
&=& \left[\left(\frac{5AB}{\epsilon}\right)^dn\right]e^{2dLt},\quad\quad\forall t\geq T.
\end{eqnarray*}

On the other hand, given $y\in K$ one has $y=f_i(v)$ for some $v\in B(0,1)$ and $1\leq i\leq n$.
For this $v$ there is $u\in \mathcal{E}(\frac{\epsilon}{e^{2Lt}AB})$ such that
$$
\|u-v\|\leq B\rho\frac{\epsilon}{e^{2Lt}AB}=\frac{\rho\epsilon}{e^{2Lt}A}.
$$
Thus, $x=f_i(u)$ belongs to $F$ and satisfies
$$
d(x,y)=d(f_i(u),f_i(v))\leq A\|u-v\|\leq A\rho\frac{\epsilon}{e^{2Lt}A}\leq \frac{\epsilon}{e^{2Lt}}\|X(x)\|.
$$
So, for all $t\geq T$ one has
$$
d(\varphi_s(x),\varphi_s(y))\leq e^{sL}d(x,y)\leq e^{tL}\frac{\epsilon}{e^{2Lt}}\|X(x)\|
\overset{\eqref{cope}}{\leq} \epsilon\|X(\varphi_s(x))\|,\quad\quad\forall 0\leq s\leq t.
$$
We conclude that $F$ is a  $(t,\epsilon,K)$-spanning set for all $t\geq T$ so
$$
R(t,\epsilon,K)\leq card(F)\leq  \left[\left(\frac{5AB}{\epsilon}\right)^dn\right]e^{2dLt},\quad\quad\forall t\geq T.
$$
Then,
$$
\limsup_{t\to\infty}\frac{1}t\log R(t,\epsilon,K)\leq 2dL.
$$
Letting $\epsilon\to0$ we get $E(X,K)\leq 2dL$
and taking
the supremum over the compact subsets $K\subset \Lambda^*$ we get \eqref{pi}
from \eqref{fria}. This completes the proof.
\end{proof}

We can also compute $E(X,\Lambda)$ through separating sets:
Given a compact $K\subset \Lambda^*$, $t\geq0$ and $\epsilon>0$ we say that $E$ is a {\em  $(t,\epsilon,K)$-separating set} if
$E\subset K$ and $B(x,t,\epsilon)\cap E=\{x\}$ for all $x\in E$.
Define
$$
S(t,\epsilon,K)=\max\{card(E)\mid E\mbox{ is a  $(t,\epsilon,K)$-separating set}\},
$$
for all $t\geq0,\, \epsilon>0.$

\begin{lem}
\label{rasca}
One has
$$
E(X,\Lambda)=\sup_K\lim_{\epsilon\to0}\limsup_{t\to\infty}\frac{1}t\log S(t,\epsilon,K),
$$
where the supremum is over the compact subsets $K\subset \Lambda^*$.
\end{lem}

\begin{proof}
First we prove that for every compact $K\subset \Lambda^*$ there is
$\epsilon_K>0$ such that
$$
S(t,\epsilon,K)\leq R(t,\frac{\epsilon}4,\Lambda),\quad\quad\forall t>0,\, 0<\epsilon<\epsilon_K.
$$
Fix such a $K$.
Then, there is $0<\epsilon_K<r_0$ such that
$$
K\subset \Lambda_{\frac{\epsilon}4},\quad\quad\forall 0<\epsilon <\epsilon_K,
$$
where
$r_0$ is given by Lemma \ref{symmetry}.

Let $E$ and $F$ be a  $(t,\epsilon,K)$-separating set and a  $(t,\frac{\epsilon}4,K)$-spanning set respectively.
Then, $E\subset K\subset M_{\frac{\epsilon}4}$ and so there is a map $\phi:E\to F$ such that
$$
x\in B(\phi(x),t,\frac{\epsilon}4),\quad\quad\forall x\in E.
$$
Now, suppose $\phi(x)=\phi(x')$ for some $x,x'\in E$.
Then, the common value $z=\phi(x)=\phi(x')$ satisfies
$x,x'\in B(z,t,\frac{\epsilon}4)$. Since $\epsilon<r_0$,
$B(z,t,\frac{\epsilon}4)\subset B(x,t,\epsilon)$ by Lemma \ref{symmetry}
so $x'\in B(x,t,\epsilon)\cap E=\{x\}$ hence $x'=x$.
It follows that $\phi$ is injective thus $card(E)\leq card(F)$. Since $E$ and $F$ are arbitrary, $S(t,\epsilon,K)\leq R(t,\frac{\epsilon}4,K)$. It follows that
\begin{equation}
\label{veneno}
S(t,\epsilon,K)<\infty,\quad\forall t\geq0,\, 0<\epsilon< r_0\mbox{ and all compact }K\subset M^*.
\end{equation}
Moreover,
$$
\limsup_{\epsilon\to0}\limsup_{t\to\infty}\frac{1}t\log S(t,\epsilon,K)\leq \lim_{\epsilon\to0}\limsup_{t\to\infty}\frac{1}t\log R(t,\frac{\epsilon}4,K)=E(X).
$$
Since $K\subset M^*$ is arbitrary,
$$
\sup_K\limsup_{\epsilon\to0}\limsup_{t\to\infty}\frac{1}t\log S(t,\epsilon,K)\leq E(X,\Lambda).
$$

To prove the reversed inequality we first prove that there is $\epsilon_0>0$ such that
for every $0<\epsilon<\epsilon_0$ there is $K_\epsilon\subset M^*$ compact such that
\begin{equation}
\label{metro}
 R(t,\epsilon,K)\leq S(t,\frac{\epsilon}4,K_\epsilon,K),\quad\quad\forall t>0.
\end{equation}
Indeed, just take $\epsilon_0=r_0$ from Lemma \ref{symmetry} and $K_\epsilon=\Lambda_\epsilon$ for the given $0<\epsilon<\epsilon_0$.
Let $E$ be a $(t,\frac{\epsilon}4,K_\epsilon)$-separating set of maximal cardinality (which is finite by \eqref{veneno}). We shall prove that $E$ is  $(t,\epsilon)$-spanning i.e.
\begin{equation}
\label{rumble}
y\in B(E,t,\epsilon),
\end{equation}
for all $y\in K_\epsilon$.

Take $y\in K_\epsilon$.
If $y\in E$ so \eqref{rumble} holds hence we can assume
$y\notin E$. It follows that $card(E\cup \{y\})=card(E)+1>card(E)$.
Since $E\cup \{y\}\subset K_\epsilon$,
we conclude that $E\cup \{y\}$ is not  $(t,\frac{\epsilon}4,K_\epsilon)$-separating. Thus, there is $z\in E\cup \{y\}$ such that
$$
(B(z,t,\frac{\epsilon}4)\cap E)\cup (B(z,t,\frac{\epsilon}4)\cap \{y\})\neq \{z\}.
$$

If $z\in E$, then $B(z,t,\frac{\epsilon}4)\cap E=\{z\}$ so $B(z,t,\frac{\epsilon}4)\cap \{y\}\neq\emptyset$ thus
$y\in B(z,t,\frac{\epsilon}4)$ hence \eqref{rumble} holds.

If $z\notin E$, then $z=y$ so $B(z,t,\frac{\epsilon}4)\cap \{y\}=\{z\}$ thus
$B(y,t,\frac{\epsilon}4)\cap E\neq\emptyset$ yielding
$x\in B(y,t,\frac{\epsilon}4)$ for some $x\in E$.
Since $\epsilon<r_0$, $B(y,t,\frac{\epsilon}4)\subset B(x,t,\epsilon)$
Lemma \ref{symmetry} and since $y\in B(y,t,\frac{\epsilon}4)$ we get
$y\in B(x,t,\epsilon)$  thus \eqref{rumble} holds.
Then, $E$ is rescaling $(t,\epsilon)$-spanning proving \eqref{metro}.

So,
\begin{eqnarray*}
\limsup_{t\to\infty}\frac{1}t\log R(t,\epsilon)&\leq& \limsup_{t\to\infty}\frac{1}t\log S(t,\frac{\epsilon}4,K_\epsilon)\\
&\leq& \lim_{\gamma\to0}\limsup_{t\to\infty}\frac{1}t\log S(t,\gamma,K_\epsilon)\\
&\leq & \sup_K\lim_{\gamma\to0}\limsup_{t\to\infty}\frac{1}t\log S(t,\gamma,K)
\end{eqnarray*}
Letting $\epsilon\to0$ above we get
$$
E(X,\Lambda)\leq \sup_K\lim_{\gamma\to0}\limsup_{t\to\infty}\frac{1}t\log S(t,\gamma,K)
$$
completing the proof.
\end{proof}

We say that a singularity $\sigma$ of $X$ is {\em dynamically isolated} if there is a neighborhood $U$ of $\sigma$ such that
$$
\bigcap_{t\in\mathbb{R}}\varphi_t(U)=\{\sigma\}.
$$
On the other hand, $\Lambda$ is {\em rescaling expansive} (or $X$ is rescaling expansive on $\Lambda$ c.f.\cite{wew}) if for every $\epsilon>0$ there is $\delta>0$ such that if $x,y\in \Lambda$ satisfy
$d(\varphi_s(x),\varphi_{h(s)}(y))\leq\delta\|X(\varphi_s(x))\|$ for all $s\in\mathbb{R}$ and some increasing homeomorphism $h:\mathbb{R}\to\mathbb{R}$, then
$\varphi_{h(0)}(y)=\varphi_{s^*}(x)$ for some $s^*\in [-\epsilon,\epsilon]$.

Denote by $N_t(X,\Lambda)$ the number of periodic orbits of period $\leq t$ contained in $\Lambda$.

\begin{lem}
\label{furo}
If $\Lambda$ rescaling expansive and every singularity in $\Lambda$ is dynamically isolated, then
$$
\limsup_{t\to\infty}\frac{1}t\log N_t(X,\Lambda)<\infty.
$$
\end{lem}

\begin{proof}
Since every singularity in $\Lambda$ is dynamically isolated,
there is $\delta>0$ such that every periodic orbit of $X$ in $\Lambda$ intersects $\Lambda_\delta$.
We now follow closely the proof of Theorem 5 in \cite{bw} with the aid of Theorem 1.1 in \cite{wy}.

Given $t,\beta>0$ let $v_\beta(t)$ denote the number of different periodic orbits in $\Lambda$ with periods belonging to the closed interval $[t-\beta,t+\beta]$.
Let $\alpha>0$ be given by Theorem 1.1-(v) in \cite{wy} for $\epsilon=1$.

We claim that
\begin{equation}
\label{poroto}
v_{\frac{\alpha}2}(t)\leq S(t,\alpha,\Lambda_\delta),\quad\quad\forall t>0.
\end{equation}

Indeed, fix $t>0$.
By selecting one point for each periodic orbit whose period belongs to the closed interval $[t-\frac{\alpha}2,t+\frac{\alpha}2]$ we form a subset $E\subset M^*$ such that $card(E)=v_{\frac{\alpha}2}(t)$.
Since every periodic orbit intersects $\Lambda_\delta$, we can further assume $E\subset \Lambda_\delta$.

Let us prove that $E$ is  $(t,\alpha,\Lambda_\delta)$-separating.
Otherwise, there would exist distinct $x,y\in E$ such that
$$
\frac{d(\varphi_s(x),\varphi_s(y))}{\|X(\varphi_s(x))\|}<\alpha,\quad\quad\forall 0\leq s\leq t.
$$
Let $a$ and $b$ be the periods of $x$ and $y$ respectively so
$a,b\in [t-\frac{\alpha}2,t+\frac{\alpha}2]$, $\varphi_a(x)=x$ and $\varphi_b(y)=y$.
Define $m=[\frac{t-\frac{\alpha}2}\alpha]$ and the sequences
$(t_i)_{i\in\mathbb{Z}},(u_i)_{i\in\mathbb{Z}}$ by
$$
t_i=pa+q\alpha\quad\mbox{ and }\quad u_i=pb+q\alpha
$$
whenever $i=pm+q$ for some $p\in \mathbb{Z}$ and $0\leq q<m.$
Since $0\leq q\alpha\leq t$ for $0\leq q<m$,
it follows that
$$
\frac{d(\varphi_{t_i}(x),\varphi_{u_i}(y))}{\|X(\varphi_{t_i}(x))\|}=\frac{d(\varphi_{q\alpha}(x),\varphi_{q\alpha}(y))}{\|X(\varphi_{q\alpha}(x))\|}<\alpha,\quad\quad\forall i\in\mathbb{Z}.
$$
Then, Theorem 1.1-(v) in \cite{wy} provides $t\in [-1,1]$ such that $\varphi_{u_0}(y)=\varphi_t(\varphi_{t_0}(x))$. It follows that
$x$ and $y$ belong to the same orbit. However, $x$ and $y$ are distinct and so they are in different orbits by construction, a contradiction. This contradiction proves that $E$ is $(t,\alpha,\Lambda_\delta)$-separating. It follows that
$card(E)\leq S(t,\alpha,\Lambda_\delta))$ and, since $card(E)=v_{\frac{\alpha}2}(t)$, we get \eqref{poroto}.

It then follows that
$$
N_t(X,\Lambda)\leq \sum_{n=1}^{[\frac{t}\alpha]}v_{\frac{\alpha}2}(n\alpha)\overset{\eqref{poroto}}{\leq} \sum_{n=1}^{[\frac{t}\alpha]} S(n\alpha,\alpha,\Lambda_\delta),\quad\quad\forall t\geq0.
$$
Now, $n\alpha\leq t$ whenever $n=1,2,\cdots, [\frac{t}\alpha]$
and $S(s,\alpha,\Lambda_\alpha)$ does not decreases as $s$ increases so
$S(n\alpha,\alpha,\Lambda_\delta)\leq S(t,\alpha,\Lambda_\delta)$ for all such $n$'s thus
$$
v(t,\Lambda)\leq \frac{t}\alpha S(t,\alpha,\Lambda_\delta),\quad\quad\forall t\geq0.
$$
Therefore,
\begin{eqnarray*}
\limsup_{t\to\infty}\frac{1}t\log N_t(X,\Lambda)&\leq& \limsup_{t\to\infty}\left(\frac{\log t}t-\frac{\log\alpha}t+\frac{1}t\log S(t,\alpha,\Lambda_\delta)\right)\\
&=&\limsup_{t\to\infty}\frac{1}t\log S(t,\alpha,\Lambda_\delta)\\
&\leq& E(X,\Lambda) \quad\quad\quad\quad\quad\quad\mbox{(by Lemma \ref{rasca})}\\
&<&\infty\quad\quad\quad\quad\quad\quad\quad\quad\quad\mbox{(by Lemma \ref{perseguido}).}
\end{eqnarray*}
\end{proof}

\begin{proof}[Proof of the theorem]
By Theorem 3 in \cite{bdl} there is an open and dense subset $\mathcal{O}$ of star flows of $M$ such that if $X\in \mathcal{O}$, then the set of periodic orbits of $X$ is contained in the union of finitely many pairwise disjoint multisingular hyperbolic sets $\Lambda$ of $X$.
It suffices to prove that $N_t(X,\Lambda)$ growths at most exponentially for all such sets $\Lambda$.
By Theorem A in \cite{wew} each $\Lambda$ is rescaling expansive. Since every singularity is hyperbolic, we also have that every singularity in $\Lambda$ is dynamically isolated.
Then,
$$
c_0=1+\limsup_{t\to\infty}\frac{1}t\log N_t(X,\Lambda)<\infty
$$
by Lemma \ref{furo}.
It follows from the definitions that
there is $T>0$ such that
$$
N_t(X,\Lambda)<c_0^t,\qquad\forall t\geq T.
$$
Then,
$$
c=\max\left\{c_0,\max_{0<t\leq T}\frac{1}t\log N_t(X,\Lambda)\right\}
$$
satisfies
$$
N_t(X,\Lambda)\leq c^t\qquad(\forall t>0)
$$
completing the proof.
\end{proof}

\section{Final Remarks}
\label{sec3}

\noindent
The methods used in this paper apply also in another situations.
For example, every sectional-hyperbolic attractor (in the sense of \cite{mm}) is rescaling expansive and the singularities on it are hyperbolic hence dynamically isolated. Then, we have at most exponential growth of periodic orbits for such attractors by Lemma \ref{furo}.

This abstract result applies to the concrete case of the Lorenz polynomial equation in $\mathbb{R}^3$ (c.f. \cite{l}):
\[
\begin{cases}
\dot{x} = \sigma(y-x),\\
\dot{y} = x(\rho - z) - y,\\
\dot{z} = xy - \beta z,
\end{cases}
\]
More precisely, {\em we assert that the growth rate of periodic orbits for the Lorenz equation with classical parameters $\sigma=10$, $\rho=28$, and $\beta=\frac{8}{3}$ is at most exponential}.

In fact, by \cite{t} the Lorenz equation with classical parameters has a compact invariant set $\Lambda$ containing all periodic orbits (this set is the geometric Lorenz attractor \cite{gw}).
Since this attractor is sectional-hyperbolic, we can apply the above abstract result to conclude the assertion.

Another proof of this assertion was provided to us by Professor D. Turaev:
the number of periodic orbits of period less than $T$ is not larger than the number of periodic orbits of the return map to the cross-section $z=r-1$ of period less than $T/\tau$, where $\tau$ is the minimal return time to the cross-section. Since the return map is hyperbolic (by Tucker's computer assisted proof \cite{t}), its periodic orbits are uniquely defined by their codings (sequences of $(-1)$'s and $1$'s corresponding to negative or positive values of  $x$, so the number of orbits of period  $n$  for this map is no more than  $2^n$ (the number of all possible codings).

\section*{Declaration of competing interest}
\label{Three}

\noindent
There is no competing interest.

\section*{Data availability}
\label{Four}

\noindent
No data was used for the research described in the article.

\section*{Funding}
XW was partially supported by NSFC12071018 and the Fundamental Research Funds for the Central Universities.
This research is part of a project that has received funding from
the European Union's European Research Council Marie Sklodowska-Curie Project No. 101151716 -- TMSHADS -- HORIZON--MSCA--2023--PF--01.


\begin{thebibliography}{10}








\bibitem{am}
Artin, M., Mazur, B.,
On periodic points,
{\em Ann. of Math (2)}, 81 (1965), 82--99.






\bibitem{bdl}
Bonatti, C., da Luz, A.,
Star flows and multisingular hyperbolicity,
{\em J. Eur. Math. Soc.} 23 (2021), 2649--2705.












\bibitem{bw}
Bowen, R., Walters, P.,
Expansive one-parameter flows.
{\em J. Differential Equations} 12 (1972), 180--193.








\bibitem{gw}
Gan, S., Wen, L.,
Nonsingular star flows satisfy Axiom A and the no-cycle condition,
{\em Invent. Math.} 164 (2006), 279--315.














\bibitem{hw}
Han, B., Wen, X.,
A shadowing lemma for quasi-hyperbolic strings of flows,
{\em J. Differential Equations} 264 (2018), 1--29.


\bibitem{ka}
Kaloshin, V.,
An extension of Artin-Mazur theorem,
{\em Ann. of Math. (2)} 150 (1999), 729--741.















\bibitem{lms}
Lee, K., Morales, C.A., San Martin, B.,
Measure $N$-expansive systems,
{\em J. Differential Equations} 267 (2019), 2053--2082.



\bibitem{lr}
Lee, K, Rojas, A.,
Eventually expansive semiflows,
{\em Comm. Pure App. Anal.} 21 (2022), 3335--3351.



\bibitem{l}
Lorenz, E.N.,
Deterministic Nonperiodic Flow,
{\em Journal of the Atmospheric Sciences} 20 (1963), 130--141.



\bibitem{mm}
Metzger, R., Morales, C.A.,
Sectional-hyperbolic systems,
{\em Ergodic Theory Dynam. Systems} 28 (2008), 1587--1597.


\bibitem{m}
Meyer, K.R.,
On the convergence of zeta function for flows and diffeomorphisms,
{\em J. Differential Equations} 5 (1969), 339--345.





\bibitem{mrw}
Morales, C.A., Rego, E., Wen, X.,
Rescaled topological entropy,
arXiv:2506.02383 [math.DS].











\bibitem{r}
Rego, E.,
{\em Entropy Theory of Expansive Systems}
Thesis, Universidade Federal do Rio de Janeiro (2021).






\bibitem{s}
Shub, M.,
Periodic orbits of hyperbolic diffeomorphisms and flows,
{\em Bull. Amer. Math. Soc.} 75 (1969), 57--58.




\bibitem{szz}
Sun, W., Zhang, C., Zhou, Y,
Extreme entropy versus growth rates of periodic orbits in equivalent flows,
{\em Topology App.} 202 (2016), 151--159.





\bibitem{t}
Tucker, W.,
A rigorous ODE solver and Smale's 14th problem,
{\em Found. Comput. Math.} 2 (2002), 53--117.







\bibitem{wew}
Wen, X., Wen, L.,
A rescaled expansiveness of flows,
{\em Trans. Amer. Math. Soc.} 371 (2019), 3179--3207.


\bibitem{wy}
Wen, X., Yu, Y.,
Equivalent definitions of rescaled expansiveness,
{\em J. Korean Math. Soc.} 55 (2018), 593--604.


\bibitem{wyz}
Wu, W., Yang, D., Zhang, Y.,
On the growth rate of periodic orbits for vector fields,
{\em Adv. Math.} 346 (2019), 170--193.



\end{thebibliography}
\end{document}